\documentclass[10pt, draft]{amsart}

\usepackage[cp1251]{inputenc}
\usepackage[T2A]{fontenc}
\usepackage{amsmath}
\usepackage{amsfonts}
\usepackage{amssymb}
\usepackage{amsthm}
\usepackage{euscript}
\usepackage[ukrainian,russian,english]{babel}
\theoremstyle{plain}
\newtheorem{theorem}{Theorem}
\newtheorem{lemma}[theorem]{Lemma}

\newtheorem{corollary}{Corollary}[theorem]
\theoremstyle{definition}

\theoremstyle{remark}

\newtheorem{remark}{Remark}[theorem]
\newtheorem{example}{Example}
\theoremstyle{plain}
\newtheorem*{theorem*}{Theorem}
\newtheorem*{lemma*}{Lemma}
\newtheorem*{proposition*}{Proposition}
\newtheorem*{statement*}{Statement}
\newtheorem*{corollary*}{Corollary}
\theoremstyle{definition}
\newtheorem*{definition*}{Definition}
\theoremstyle{remark}
\newtheorem*{notation*}{Notation}
\newtheorem*{remark*}{Remark}
\newtheorem*{example*}{Example}
\renewcommand{\theexample}%
{\Alph{example}}

\begin{document}
%\footnotetext[1]{The investigation is partially supported by DFFD of Ukraine under grant $28.1/017$.}

\title[Smoothness of Hill's potential and lengths of spectral gaps]{Smoothness of Hill's potential and lengths \\ of spectral gaps$^{\ast}$}

\address{Institute of Mathematics of NAS of Ukraine \\
         Tereshchenkivska str., 3 \\
         Kyiv-4 \\
         Ukraine \\
         01601}

\author[V.Mikhailets, V. Molyboga] {Vladimir Mikhailets, Volodymyr Molyboga}

\email[Vladimir Mikhailets]{mikhailets@imath.kiev.ua}

\email[Volodymyr Molyboga]{molyboga@imath.kiev.ua}

\thanks{$^{\ast}$The investigation is partially supported by DFFD of Ukraine under grant $28.1/017$.}



%\title{Hill's potentials in H\"{o}rmander spaces and their spectral gaps\tnoteref{t1}} \tnotetext[t1]{The investigation is partially
%supported by DFFD of Ukraine under grant $\Phi28.1/017$.}

%\author[rvt]{Vladimir Mikhailets}
%\ead{mikhailets@imath.kiev.ua}

%\author[rvt]{Volodymyr Molyboga\corref{cor}}
%\ead{molyboga@imath.kiev.ua}

%\cortext[cor]{Corresponding author}

%\address[rvt]{Institute of Mathematics of NAS of Ukraine,
%Tereshchenkivska str. 3, Kyiv-4, Ukraine, 01601}

%\begin{keyword}
%Hill-Schr\"{o}dinger operators \sep singular potentials \sep
%spectral gaps \sep H\"{o}rmander spaces

%\MSC Primary 34L40 \sep Secondary 47A10 \sep 47A75
%\end{keyword}

%%%%%%%%%%%%%%%%%%%%%%%%%%%%%%%%%%%%%%%%%%%%%%%%%%%%%%%%%%%%%%%%%%%%%%%%%%%%%%%%%%%%%%%%%%%%%%%%%%%%%%%%%%%%%%%%%%%%%%%%%%%%%%%%%%%%%%%%%%%%%%%%%%%%%%%%%%%%%%%%%%%%%%%%%
%%%%%%%%%%%%%%%%%%%%%%%%%%%%%%%%%%%%%%%%%%%%%%%%%%%%%%%%%%%%%%%%%%%%%%%%%%%%%%%%%%%%%%%%%%%%%%%%%%%%%%%%%%%%%%%%%%%%%%%%%%%%%%%%%%%%%%%%%%%%%%%%%%%%%%%%%%%%%%%%%%%%%%%%%

\begin{abstract}
Let $\left\{\gamma_q(n)\right\}_{n \in \mathbb{N}}$ be the lengths of spectral gaps in a continuous spectrum of the Hill-Schr\"{o}dinger operators
\begin{equation*}
  S(q)u=-u''+q(x)u,\quad x\in \mathbb{R},
\end{equation*}
with 1-periodic real-valued potentials $q \in L^{2}\left(\mathbb{T}\right)$. Let weight function $\omega:\;[1,\infty)\rightarrow (0,\infty)$. We prove that under the
condition
\begin{equation*}
  \exists s\in [0,\infty):\quad k^{s}\ll\omega(k)\ll k^{s+1},\; k\in \mathbb{N},
\end{equation*}
the map $\gamma:\, q \mapsto \left\{\gamma_{q}(n)\right\}_{n \in \mathbb{N}}$ satisfies the equalities:
\begin{equation*}
  \verb"i")\quad  \gamma\left(H^{\omega} \right) = h_{+}^{\omega},\hspace{30pt}
  \verb"ii")\quad  \gamma^{-1}\left(h_{+}^{\omega}\right) = H^{\omega},
\end{equation*}
where the real function space
\begin{align*}
  H^{\omega} & =\left\{f=\sum_{k\in \mathbb{Z}}\widehat{f}\,(k)e^{i k2\pi x}\in L^{2}\left(\mathbb{T}\right)\left|\;
  \sum_{k\in \mathbb{N}}  \omega^{2}(k)|\widehat{f}(k)|^{2}<\infty,\; \widehat{f}(k)=\overline{\widehat{f}(-k)},\;k\in \mathbb{Z}\right.\right\},
\end{align*}
and
\begin{equation*}
  h^{\omega} = \left\{a=\{a(k)\}_{k\in \mathbb{N}}\left|\sum_{k\in \mathbb{N}}\omega^{2}(k)|a(k)|^{2}<\infty\right.\right\},  \,\,
  h_{+}^{\omega} = \left\{a=\{a(k)\}_{k\in \mathbb{N}}\in h^{\omega}\left|\; a(k)\geq 0\right.\right\}.
\end{equation*}
If the weight $\omega$ is such that
\begin{equation*}
  \exists a>1,c>1:\qquad c^{-1}\leq \frac{\omega(\lambda t)}{\omega (t)}\leq c\quad\forall t\geq 1,\;\lambda\in [1,a]
\end{equation*}
then the function class $H^{\omega}$ is a real H\"{o}rmander space $H_{2}^{\omega}(\mathbb{T},\mathbb{R})$ with the weight $\omega(\sqrt{1+\xi^{2}})$.
\end{abstract}

\maketitle

%%%%%%%%%%%%%%%%%%%%%%%%%%%%%%%%%%%%%%%%%%%%%%%%%%%%%%%%%%%%%%%%%%%%%%%%%%%%%%%%%%%%%%%%%%%%%%%%%%%%%%%%%%%%%%%%%%%%%%%%%%%%%%%%%%%%%%%%%%%%%%%%%%%%%%%%%%%%%%%%%%%%%%%%%
%%%%%%%%%%%%%%%%%%%%%%%%%%%%%%%%%%%%%%%%%%%%%%%%%%%%%%%%%%%%%%%%%%%%%%%%%%%%%%%%%%%%%%%%%%%%%%%%%%%%%%%%%%%%%%%%%%%%%%%%%%%%%%%%%%%%%%%%%%%%%%%%%%%%%%%%%%%%%%%%%%%%%%%%%
\section{Introduction}\label{sct_Int}
Let consider on the complex Hilbert space $L^{2}(\mathbb{R})$ the Hill-Schr\"{o}dinger operators
\begin{equation}\label{eq_10}
  S(q)u:=-u''+q(x)u,\quad x\in \mathbb{R},
\end{equation}
with 1-periodic real-valued potentials
\begin{equation*}
  q(x)=\sum_{k\in \mathbb{Z}}\widehat{q}(k)e^{i k 2\pi x}\in L^{2}(\mathbb{T},\mathbb{R}),\quad
  \mathbb{T}:=\mathbb{R}/\mathbb{Z}.
\end{equation*}
Last condition means that
\begin{equation*}\label{eq_11}
 \sum_{k\in \mathbb{Z}}|\widehat{q}(k)|^{2}<\infty
 \quad\text{and}\quad \widehat{q}(k)=\overline{\widehat{q}(-k)},\quad k\in \mathbb{Z}.
\end{equation*}

It is well known that the operators $S(q)$ are lower semibounded and self-adjoint. Their spectra are absolutely continuous and have a zone
structure \cite{ReSi}.

Spectra of the operators $S(q)$ are completely defined by the location of the endpoints of spectral gaps
$\{\lambda_{0}(q),\lambda_{n}^{\pm}(q)\}_{n=1}^{\infty}$, which satisfy the inequalities:
\begin{equation}\label{InEq}
  -\infty<\lambda_{0}(q)<\lambda_{1}^{-}(q)\leq\lambda_{1}^{+}(q)<\lambda_{2}^{-}(q)\leq\lambda_{2}^{+}(q)<\cdots\,.
\end{equation}
For even/odd numbers $n\in \mathbb{Z}_{+}$ the endpoints of spectral gaps $\{\lambda_{0}(q),\lambda_{n}^{\pm}(q)\}_{n=1}^{\infty}$ are eigenvalues
of the periodic/semiperiodic problems on the interval $[0,1]$:
\begin{align*}
   S_{\pm}(q)u & :=-u''+q(x)u=\lambda u, \\
   \mathrm{Dom}(S_{\pm}(q)) & :=\left\{u\in H^{2}[0,1]\left|\, u^{(j)}(0)=\pm\, u^{(j)}(1),\, j=0,1\right.\right\}.\hspace{130pt}
\end{align*}

Interiors of spectral bands (stability or tied zones)
\begin{equation*}
  \mathcal{B}_{0}(q):=(\lambda_{0}(q),\lambda_{1}^{-}(q)),\qquad
  \mathcal{B}_{n}(q):=(\lambda_{n}^{+}(q),\lambda_{n+1}^{-}(q)),\quad n\in
  \mathbb{N},
\end{equation*}
together with the \textit{collapsed} gaps
\begin{equation*}
 \lambda=\lambda_{n}^{+}=\lambda_{n}^{-}, \quad n\in \mathbb{N}
\end{equation*}
are characterized as a locus of those real $\lambda\in \mathbb{R}$ for which all solutions of the equation $S(q) u=\lambda u$ are bounded.
\textit{Open} spectral gaps (instability or forbidden zones)
\begin{equation*}
  \mathcal{G}_{0}(q):=(-\infty,\lambda_{0}(q)),\qquad
  \mathcal{G}_{n}(q):=(\lambda_{n}^{-}(q),\lambda_{n}^{+}(q))\neq\emptyset,\quad n\in
  \mathbb{N}
\end{equation*}
are a locus of those real $\lambda\in \mathbb{R}$ for which any nontrivial solution of the equation $S(q) u=\lambda u$ is unbounded.

We study the behaviour of the lengths of spectral gaps
\begin{equation*}
  \gamma_{q}(n):=\lambda_{n}^{+}(q)-\lambda_{n}^{-}(q),\quad n\in \mathbb{N}
\end{equation*}
of the Hill-Schr\"{o}dinger operators $S(q)$ in terms of the behaviour of the Fourier coefficients $\{\widehat{q}(n)\}_{n\in \mathbb{N}}$ of the
potentials $q$ with respect to appropriate weight spaces, that is in terms of potential regularity.

Hochstadt \cite{Hchs1, Hchs2}, Marchenko and Ostrovskii \cite{MrOs}, McKean and Trubowitz \cite{McKTr} proved that the potential $q(x)$ is an infinitely differentiable
function if and only if the lengths of spectral gaps $\{\gamma_{q}(n)\}_{n=1}^{\infty}$ decrease faster than arbitrary power of $1/n$:
\begin{equation*}
  q(x)\in C^{\infty}(\mathbb{T},\mathbb{R})\Leftrightarrow
  \gamma_{q}(n)=O(n^{-k}),\; n\rightarrow\infty\quad \forall  k\in \mathbb{Z}_{+}.
\end{equation*}

Marchenko and Ostrovskii \cite{MrOs} (see also \cite{Mrch}) discovered that:
\begin{equation}\label{eq_14}
  q\in H^{s}(\mathbb{T},\mathbb{R})\Leftrightarrow\sum_{n\in \mathbb{N}}(1+2n)^{2s}\gamma_{q}^{2}(n),\qquad s\in
  \mathbb{Z}_{+},
\end{equation}
where $H^{s}(\mathbb{T},\mathbb{R})$, $s\in \mathbb{Z}_{+}$, denotes the Sobolev space of 1-periodic real-valued functions on the circle~$\mathbb{T}$.

To characterize regularity of potentials in the finer way we shall use the real H\"{o}rmander spaces $ H^{\omega}(\mathbb{T},\mathbb{R})$ where $\omega(\cdot)$ is a positive weight (see Appendix).
In the case of the Sobolev spaces it is a power one.

Djakov, Mityagin \cite{DjMt2}, P\"{o}schel \cite{Psch} extended the Marchenko-Ostrovskii Theorem \eqref{eq_14} to the general class of
weights $\Omega=\{\Omega(k)\}_{k\in \mathbb{N}}$ satisfying the following conditions:
\begin{align*}
 \verb"i")\hspace{5pt} & \Omega(k)\nearrow\infty,\; k\in \mathbb{N};\hspace{5pt}\text{(monotonicity)}  \hspace{220pt} \\
\verb"ii")\hspace{5pt} & \Omega(k+m)\leq \Omega(k)\Omega(m)\quad k,m\in \mathbb{N};\hspace{5pt}\text{(submultiplicity)} \\
\verb"iii")\hspace{5pt} & \frac{\log\Omega(k)}{k}\searrow 0,\quad k\rightarrow\infty,\hspace{5pt}\text{(subexponentiality)}.
\end{align*}
For such weights they proved that
\begin{equation}\label{eq_16}
  q\in H^{\Omega}(\mathbb{T},\mathbb{R})\Leftrightarrow \{\gamma_{q}(\cdot)\}\in  h^{\Omega}(\mathbb{N}).
\end{equation}
Here $h^{\Omega}(\mathbb{N})$ is the Hilbert space of weighted sequences generated by the weight $\Omega(\cdot)$.

Earlier Kappeler, Mityagin \cite{KpMt2} proved the direct implication in \eqref{eq_16} under the only assumption of submultiplicity. In the special cases of the
Abel-Sobolev weights, the Gevrey weights and the slowly increasing weights the relationship \eqref{eq_16} was established by Kappeler, Mityagin \cite{KpMt1}
($\Rightarrow$) and Djakov, Mityagin \cite{DjMt, DjMt1} ($\Leftarrow$). Detailed exposition of these results is given in the survey \cite{DjMt2}. It should be noted that
Kappeler, Mityagin \cite{KpMt1, KpMt2}, Djakov, Mityagin \cite{DjMt1, DjMt2} and P\"{o}schel \cite{Psch} studied also the more general case of complex-valued
potentials.

%%%%%%%%%%%%%%%%%%%%%%%%%%%%%%%%%%%%%%%%%%%%%%%%%%%%%%%%%%%%%%%%%%%%%%%%%%%%%%%%%%%%%%%%%%%%%%%%%%%%%%%%%%%%%%%%%%%%%%%%%%%%
%\newpage
%%%%%%%%%%%%%%%%%%%%%%%%%%%%%%%%%%%%%%%%%%%%%%%%%%%%%%%%%%%%%%%%%%%%%%%%%%%%%%%%%%%%%%%%%%%%%%%%%%%%%%%%%%%%%%%%%%%%%%%%%%%%
\section{Main result}\label{sct_MnRs}
The main purpose of this paper is to prove the following result.
\begin{theorem}\label{th_10}
Let $q\in L^{2}(\mathbb{T},\mathbb{R})$ and the weight $\omega=\{\omega(k)\}_{k\in\mathbb{N}}$ satisfy conditions:
\begin{equation*}\label{eq_21}
  k^{s}\ll \omega(k)\ll k^{1+s},\qquad s\in [0,\infty).
\end{equation*}
Then the map $\gamma:\, q \mapsto \left\{\gamma_{q}(n)\right\}_{n \in \mathbb{N}}$ satisfies the equalities:
\begin{align*}
  \verb"i")\quad & \gamma\left(H^{\omega}(\mathbb{T},\mathbb{R})\right) = h_{+}^{\omega}(\mathbb{N}),\hspace{275pt} \\
  \verb"ii")\quad & \gamma^{-1}\left(h_{+}^{\omega}(\mathbb{N})\right) = H^{\omega}(\mathbb{T},\mathbb{R}).
\end{align*}
\end{theorem}
Theorem \ref{th_10} immediately derives following corollary.
\begin{corollary}\label{cr_10}
Let for the weight $\omega=\{\omega(k)\}_{k\in\mathbb{N}}$ exist the order
\begin{equation*}
  \lim_{k\rightarrow\infty}\frac{\log\omega(k)}{\log k}=s\in [0,\infty),
\end{equation*}
and let for $s = 0$ the values of the weight $\omega=\{\omega(k)\}_{k\in\mathbb{N}}$ be separated from zero.
Then
\begin{equation*}
  q\in H^{\omega}(\mathbb{T},\mathbb{R})\Leftrightarrow \{\gamma_{q}(\cdot)\}\in  h^{\omega}(\mathbb{N}).
\end{equation*}
\end{corollary}

From Corollary \ref{cr_10} we receive the following result.
\begin{corollary}[\cite{MiMl}]\label{cr_12}
Let the weight  $\omega=\{\omega(k)\}_{k\in\mathbb{N}}$ be a regular varying sequence in the Karamata sense with the index $s\in [0,\infty)$, and
let for $s = 0$ its values be separated from zero.
Then
\begin{equation*}
  q\in H^{\omega}(\mathbb{T},\mathbb{R})\Leftrightarrow \{\gamma_{q}(\cdot)\}\in  h^{\omega}(\mathbb{N}).
\end{equation*}
\end{corollary}

Note that the assumption of Corollary \ref{cr_12} holds, for instance, for the weight
\begin{align*}
  \omega(k) & =(1+2k)^{s}\,(\log (1+k))^{r_{1}}(\log\log (1+k))^{r_{2}}\ldots (\log\log\ldots\log (1+k))^{r_{p}}, \\
    s & \in(0,\infty),\;\{r_{1},\ldots,r_{p}\}\subset \mathbb{R},\;p\in \mathbb{N},\hspace{210pt}
\end{align*}
see \cite{BnGlTg}.

The following Example \ref{ex_10} shows that statement \eqref{eq_16} does not cover Corollary \ref{cr_10} and
moreover Theorem \ref{th_10}.
\begin{example}\label{ex_10}
Let $s\in [0,\infty)$. Set
\begin{equation*}
  w(k):=
  \begin{cases}
    k^{s}\log(1+k) & \text{if}\quad k\in 2\mathbb{N}, \\
    k^{s} & \text{if}\quad k\in (2\mathbb{N}-1).
  \end{cases}
\end{equation*}
Then the weight $\omega=\{\omega(k)\}_{k\in \mathbb{N}}$ satisfies the conditions of Corollary \ref{cr_10}. But one can prove that it is not equivalent to any monotonic
weight.
\end{example}

\begin{remark} Theorem \ref{th_10} shows that if the sequence $\{|\widehat{q}(n_k)|\}_{k=1}^{\infty}$
decreases particularly fast on a certain subsequence $\{n_{k}\}_{k = 1}^{\infty} \subset \mathbb{N}$,
then so does sequence $\{\gamma_q(n_k)\}_{k=1}^{\infty}$ on the same subsequence.
Inverse statement is also true.
\end{remark}

%
%\begin{example}\label{ex_12}
%Let
%\begin{align*}
%  \verb"i")\hspace{5pt} & s\in (-1,\infty),\hspace{380pt} \\
%  \verb"ii")\hspace{5pt} & 0< a <
%  \begin{cases}
%    1+s & \text{if}\quad s\in (-1,0], \\
%    1 & \text{if}\quad  s\in [0,\infty).
%  \end{cases}
%  \\
%  \verb"iii")\hspace{5pt} & \mathbb{N}:=\mathbb{K}_{1}\bigcup\mathbb{K}_{2},\quad \sharp\mathbb{K}_{1}=\aleph_{0},\;
%  \sharp\mathbb{K}_{2}=\aleph_{0},\;\mathbb{K}_{1}\bigcap\mathbb{K}_{2}=\emptyset.
%\end{align*}
%Define the weight $\omega$ on $\mathbb{N}$ as:
%\begin{equation*}
%\omega=\{\omega(k)\}_{k\in \mathbb{N}}:\qquad  w(k):=
%  \begin{cases}
%    k^{s}\log(1+k) & \text{if}\quad k\in \mathbb{K}_{1}, \\
%    k^{s+a}\log^{-1}(1+k) & \text{if}\quad k\in \mathbb{K}_{2}.
%  \end{cases}
%\end{equation*}
%\end{example}
%%%%%%%%%%%%%%%%%%%%%%%%%%%%%%%%%%%%%%%%%%%%%%%%%%%%%%%%%%%%%%%%%%%%%%%%%%%%%%%%%%%%%%%%%%%%%%%%%%%%%%%%%%%%%%%%%%%%%%%%%%%%
%%%%%%%%%%%%%%%%%%%%%%%%%%%%%%%%%%%%%%%%%%%%%%%%%%%%%%%%%%%%%%%%%%%%%%%%%%%%%%%%%%%%%%%%%%%%%%%%%%%%%%%%%%%%%%%%%%%%%%%%%%%%
\section{Preliminaries}\label{sct_Prl}
Here, for convenience, we define Hilbert spaces of weighted two-sided sequences and formulate the Convolution Lemma \ref{lm_10}.

For every positive sequence $\omega=\{\omega(k)\}_{k\in\mathbb{N}}$ there exists its unique extension on $\mathbb{Z}$ which is a two-sided
sequence satisfying the conditions:
\begin{align*}
  \verb"i")\hspace{5pt} & \omega(0)=1;\hspace{340pt} \\
  \verb"ii")\hspace{5pt} & \omega(-k)=\omega(k)\quad \forall k\in \mathbb{N}; \\
  \verb"iii")\hspace{5pt} & \omega(k)>0\quad \forall k\in \mathbb{Z}.
\end{align*}

Let $h^{\omega}(\mathbb{Z})\equiv h^{\omega}(\mathbb{Z},\mathbb{C})$ be the Hilbert space of two-sided sequences:
\begin{align*}
  h^{\omega}(\mathbb{Z}) & :=\left\{a=\{a(k)\}_{k\in \mathbb{Z}}\left|
 \sum_{k\in \mathbb{Z}}\omega^{2}(k)|a(k)|^{2}<\infty\right.\right\}, \\
  (a,b)_{h^{\omega}(\mathbb{Z})} & :=\sum_{k\in \mathbb{Z}}\omega^{2}(k)a(k)\overline{b(k)},\quad
  a,b\in h^{\omega}(\mathbb{Z}), \\
  \|a\|_{h^{\omega}(\mathbb{Z})} & :=(a,a)_{h^{\omega}(\mathbb{Z})}^{1/2},\quad
  a\in h^{\omega}(\mathbb{Z}).
\end{align*}
By $h^{\omega}(n)$ for a convenience we will denote the $n$th element of a sequence $a=\{a(k)\}_{k\in \mathbb{Z}}$ in $h^{\omega}(\mathbb{Z})$.

Basic weights which we use are the power ones:
\begin{equation*}
 w_{s}=\left\{w_{s}(k)\right\}_{k\in \mathbb{Z}}:\qquad w_{s}(k)=(1+2|k|)^{s},\quad s\in
 \mathbb{R}.
\end{equation*}
In this case it is convenient to use shorter notations:
\begin{equation*}
   h^{\omega_{s}}(\mathbb{Z})\equiv h^{s}(\mathbb{Z}),\quad s\in \mathbb{R}.
\end{equation*}

Operation of convolution for two-sided sequences
\begin{equation*}
  a=\{a(k)\}_{k\in\mathbb{Z}}\quad\text{and}\quad b=\{b(k)\}_{k\in\mathbb{Z}}
\end{equation*}
is formally defined as follows:
\begin{align*}
  (a,b) & \mapsto a\ast b, \\
  (a\ast b)(k) & :=\sum_{j\in \mathbb{Z}}a(k-j)\,b(j),\quad k\in \mathbb{Z}.
\end{align*}

Sufficient conditions for the convolution to exist as a continuous map are given by the following known lemma, see for example \cite{KpMh, Mhr}.
\begin{lemma}[The Convolution Lemma]\label{lm_10}
Let $s,r\geq 0$, and $t\leq\min(s,r)$, $t\in \mathbb{R}$. If $s+r-t>1/2$ then the convolution $(a,b)\mapsto a\ast b$ is well defined as a
continuous map acting in the spaces:
\begin{align*}
  (a)\hspace{5pt} & h^{s}(\mathbb{Z})\times h^{r}(\mathbb{Z})\rightarrow h^{t}(\mathbb{Z}),\hspace{280pt} \\
  (b)\hspace{5pt} & h^{-t}(\mathbb{Z})\times h^{s}(\mathbb{Z})\rightarrow h^{-r}(\mathbb{Z}).
\end{align*}

In the case $s+r-t<1/2$ this statement fails to hold.
\end{lemma}
%%%%%%%%%%%%%%%%%%%%%%%%%%%%%%%%%%%%%%%%%%%%%%%%%%%%%%%%%%%%%%%%%%%%%%%%%%%%%%%%%%%%%%%%%%%%%%%%%%%%%%%%%%%%%%%%%%%%%%%%%%%%
%%%%%%%%%%%%%%%%%%%%%%%%%%%%%%%%%%%%%%%%%%%%%%%%%%%%%%%%%%%%%%%%%%%%%%%%%%%%%%%%%%%%%%%%%%%%%%%%%%%%%%%%%%%%%%%%%%%%%%%%%%%%
\section{The Proofs}\label{sct_Prf}
Basic point of our proof of Theorem \ref{th_10} is sharp asymptotic formulae for the lengths of spectral gaps $\{\gamma_{q}(n)\}_{n\in\mathbb{N}}$ of the
Hill-Schr\"{o}dinger operators $S(q)$ and fundamental result of \cite[Theorem 1]{GrTr}.
\begin{lemma}\label{lm_12}
The lengths of spectral gaps $\{\gamma_{q}(n)\}_{n\in\mathbb{N}}$ of the Hill-Schr\"{o}dinger operators $S(q)$ with $q\in
H^{s}(\mathbb{T},\mathbb{R})$, $s\in [0,\infty)$, uniformly on the bounded sets of potentials $q$ in the corresponding Sobolev spaces
$H^{s}(\mathbb{T})$ for $n\geq n_{0}$, $n_{0}=n_{0}\left(\|q\|_{H^{s}(\mathbb{T})}\right)$, satisfy the following asymptotic formulae:
\begin{equation}\label{eq_30}
  \gamma_{q}(n)=2|\widehat{q}(n)|+h^{1+s}(n).
\end{equation}
\end{lemma}
\begin{proof}[Proof of Lemma \ref{lm_12}]
To prove asymptotic formulae \eqref{eq_30} we use \cite[Theorem 1.2]{KpMt2} and the Convolution Lemma~\ref{lm_10} (see also \cite[Appendix]{KpMt2}). Indeed, applying
\cite[Theorem 1.2]{KpMt2} with $q\in H^{s}(\mathbb{T},\mathbb{R})$, $s\in [0,\infty)$, we get
\begin{equation}\label{eq_32}
  \sum_{n\in \mathbb{N}}(1+2n)^{2(1+s)}\left(\min_{\pm}\left|\gamma_{q}(n)\pm
  2\sqrt{(\widehat{q}+\varrho)(-n)(\widehat{q}+\varrho)(n)}\right|\right)^{2}\leq
  C\left(\|q\|_{H^{s}(\mathbb{T})}\right),
\end{equation}
where
\begin{equation*}
  \varrho(n):=\frac{1}{\pi^{2}}\sum_{j\in
  \mathbb{Z}\setminus\{\pm n\}}\frac{\widehat{q}(n-j)\widehat{q}(n+j)}{(n-j)(n+j)}.
\end{equation*}
Without losing generality we assume that
\begin{equation}\label{eq_33}
 \widehat{q}(0):=0.
\end{equation}

Taking into account that the potentials $q$ are real-valued we have
\begin{equation*}
  \widehat{q}(k)=\overline{\widehat{q}(-k)},\; \varrho(k)=\overline{\varrho(-k)},\quad  k\in \mathbb{Z}.
\end{equation*}
Then from \eqref{eq_32} we get the estimates
\begin{equation}\label{eq_34}
  \left\{\gamma_{n}(q)-2\left|\widehat{q}(n)+\varrho(n)\right|\right\}_{n\in \mathbb{N}}\in h^{1+s}(\mathbb{N}).
\end{equation}

Further, as  by assumption $q\in H^{s}(\mathbb{T},\mathbb{R})$, that is $\{\widehat{q}(k)\}_{k\in \mathbb{Z}}\in h^{s}(\mathbb{Z})$, then taking into account \eqref{eq_33}
\begin{equation*}
  \left\{\frac{\widehat{q}(k)}{k}\right\}_{k\in\mathbb{Z}}\in h^{1+s}(\mathbb{Z}),\quad s\in
  [0,\infty).
\end{equation*}
 Applying the Convolution Lemma \ref{lm_10} we obtain
\begin{align}\label{eq_36}
  \varrho(n) & =\frac{1}{\pi^{2}}\sum_{j\in
  \mathbb{Z}}\frac{\widehat{q}(n-j)\widehat{q}(n+j)}{(n-j)(n+j)}=\frac{1}{\pi^{2}}\sum_{j\in
  \mathbb{Z}}\frac{\widehat{q}(2n-j)}{2n-j}\cdot\frac{\widehat{q}(j)}{j}  \\
  & =\left(\left\{\frac{\widehat{q}(k)}{k}\right\}_{k\in \mathbb{Z}}\ast\left\{\frac{\widehat{q}(k)}{k}\right\}_{k\in
  \mathbb{Z}}\right)(2n)\in h^{1+s}(\mathbb{N}). \notag
\end{align}

Finally, from \eqref{eq_34} and \eqref{eq_36} we get the necessary estimates \eqref{eq_30}.

The proof of Lemma \ref{lm_12} is complete.
\end{proof}

\begin{proof}[\textbf{Proof of Theorem \ref{th_10}}.] Let $q\in L^{2}(\mathbb{T},\mathbb{R})$ and $\omega=\{\omega(k)\}_{k\in\mathbb{N}}$ be a given weight satisfying the
conditions of Theorem \ref{th_10}:
\begin{equation}\label{eq_36.1}
  k^{s}\ll \omega(k)\ll k^{1+s},\qquad s\in [0,\infty).
\end{equation}

At first, we need to prove the statement
\begin{equation}\label{eq_37}
  q\in H^{\omega}(\mathbb{T},\mathbb{R})\Leftrightarrow \{\gamma_{q}(\cdot)\}\in  h^{\omega}(\mathbb{N}).
\end{equation}

Due to formulae \eqref{eq_36.1} the continuous embeddings
\begin{align}
 H^{1+s}(\mathbb{T}) & \hookrightarrow H^{\omega}(\mathbb{T})\hookrightarrow H^{s}(\mathbb{T}),\label{eq_38} \\
 h^{1+s}(\mathbb{N}) & \hookrightarrow h^{\omega}(\mathbb{N})\hookrightarrow h^{s}(\mathbb{N}), \hspace{20pt}
 s\in [0,\infty),\hspace{100pt} \label{eq_40}
\end{align}
are valid because
\begin{equation}\label{eq_42}
  H^{\omega_{1}}(\mathbb{T})\hookrightarrow H^{\omega_{2}}(\mathbb{T}),\quad
  h^{\omega_{1}}(\mathbb{N})\hookrightarrow h^{\omega_{2}}(\mathbb{N})\quad\text{if only}\quad \omega_{1}\gg \omega_{2}.
\end{equation}

Let $q\in H^{\omega}(\mathbb{T},\mathbb{R})$, then from \eqref{eq_38} we get $q\in H^{s}(\mathbb{T},\mathbb{R})$. Due to Lemma \ref{lm_12} we
find that
\begin{equation*}
  \gamma_{q}(n)=2|\widehat{q}(n)|+h^{1+s}(n).
\end{equation*}

Applying \eqref{eq_40} from the latter we derive
\begin{equation*}
  \gamma_{q}(n)=2|\widehat{q}(n)|+h^{\omega}(n).
\end{equation*}
And, as a consequence, we obtain that $\{\gamma_{q}(\cdot)\}\in h^{\omega}(\mathbb{N})$.

Direct implication in \eqref{eq_37} has been proved.

Let $\{\gamma_{q}(\cdot)\}\in h^{\omega}(\mathbb{N})$. Applying \eqref{eq_40} we get $\{\gamma_{q}(\cdot)\}\in h^{s}(\mathbb{N})$. Further, from
\eqref{eq_16} with $\Omega(k)=(1+2k)^{s}$, $s\in [0,\infty)$, we obtain $q\in H^{s}(\mathbb{T},\mathbb{R})$.

We have already proved the implication
\begin{equation*}
  q\in H^{s}(\mathbb{T},\mathbb{R})\Rightarrow \gamma_{q}(n)=2|\widehat{q}(n)|+h^{\omega}(n).
\end{equation*}
Hence $\{\widehat{q}(\cdot)\}\in h^{\omega}(\mathbb{N})$, i.e., $q\in H^{\omega}(\mathbb{T},\mathbb{R})$.

Inverse implication in \eqref{eq_37} has been proved.

Now we are ready to prove the statement of Theorem \ref{th_10}.

From relationship \eqref{eq_37} we get
\begin{equation}\label{eq_44}
  \gamma\left(H^{\omega}(\mathbb{T},\mathbb{R})\right) \subset h_{+}^{\omega}(\mathbb{N}).
\end{equation}
To establish the equality i) of Theorem \ref{th_10} it is necessary to prove the inverse inclusion in latter formula \eqref{eq_44}. So, let $\{\gamma(n)\}_{n\in
\mathbb{N}}$ be an arbitrary sequence from the space $h_{+}^{\omega}(\mathbb{N})$. Then $\{\gamma(n)\}_{n\in \mathbb{N}}\in h_{+}^{0}(\mathbb{N})$. Due to \cite[Theorem
1]{GrTr} potential $q\in L^{2}(\mathbb{T},\mathbb{R})$ exists for which the sequence $\{\gamma(n)\}_{n\in \mathbb{N}}\in h_{+}^{0}(\mathbb{N})$ is a corresponding
sequence of the lengths of spectral gaps. As by assumption $\{\gamma(n)\}_{n\in \mathbb{N}}\in h_{+}^{\omega}(\mathbb{N})$ due to \eqref{eq_37} we conclude that $q\in
H^{\omega}(\mathbb{T},\mathbb{R})$. I.e., the inclusion
\begin{equation}\label{eq_46}
  \gamma\left(H^{\omega}(\mathbb{T},\mathbb{R})\right) \supset h_{+}^{\omega}(\mathbb{N})
\end{equation}
holds.

Finally, the inclusions \eqref{eq_44} and \eqref{eq_46} give the necessary equality i).

Now, let prove the equality ii) of Theorem \ref{th_10}. Let $\{\gamma(n)\}_{n\in \mathbb{N}}$ be an arbitrary sequence from the space $h_{+}^{\omega}(\mathbb{N})$.
Similarly as above we prove that potential $q\in H^{\omega}(\mathbb{T},\mathbb{R})$ exists for which the sequence $\{\gamma(n)\}_{n\in \mathbb{N}}\in
h_{+}^{0}(\mathbb{N})$ is a corresponding sequence of the lengths of spectral gaps. That is
\begin{equation}\label{eq_48}
  \gamma^{-1}\left(h_{+}^{\omega}(\mathbb{N})\right) \subset H^{\omega}(\mathbb{T},\mathbb{R}).
\end{equation}

Inversely. Let $q$ be an arbitrary function from the H\"{o}rmander space $H^{\omega}(\mathbb{T},\mathbb{R})$. Then due to \eqref{eq_37} we have
$\gamma_{q}=\{\gamma_{q}(n)\}_{n\in \mathbb{N}}\in h_{+}^{\omega}(\mathbb{N})$. I.e.,
\begin{equation}\label{eq_50}
  \gamma^{-1}\left(h_{+}^{\omega}(\mathbb{N})\right) \supset H^{\omega}(\mathbb{T},\mathbb{R}).
\end{equation}

The inclusions \eqref{eq_48} and \eqref{eq_50} give the equality ii) of Theorem \ref{th_10}.

The proof of Theorem \ref{th_10} is complete.
\end{proof}
%%%%%%%%%%%%%%%%%%%%%%%%%%%%%%%%%%%%%%%%%%%%%%%%%%%%%%%%%%%%%%%%%%%%%%%%%%%%%%%%%%%%%%%%%%%%%%%%%%%%%%%%%%%%%%%%%%%%%%%%%%%%%%%%%%%%%%%%%%%%%%%%%%%%%%%%%%%%%%%%%%%%%%%%%
%%%%%%%%%%%%%%%%%%%%%%%%%%%%%%%%%%%%%%%%%%%%%%%%%%%%%%%%%%%%%%%%%%%%%%%%%%%%%%%%%%%%%%%%%%%%%%%%%%%%%%%%%%%%%%%%%%%%%%%%%%%%%%%%%%%%%%%%%%%%%%%%%%%%%%%%%%%%%%%%%%%%%%%%%
%\section*{Acknowledgments}
%The authors would like to express their gratitude to Professors V.A.~Mar\-chen\-ko and F.S.~Rofe-Beketov, for their attention to this paper.
%%%%%%%%%%%%%%%%%%%%%%%%%%%%%%%%%%%%%%%%%%%%%%%%%%%%%%%%%%%%%%%%%%%%%%%%%%%%%%%%%%%%%%%%%%%%%%%%%%%%%%%%%%%%%%%%%%%%%%%%%%%%%%%%%%%%%%%%%%%%%%%%%%%%%%%%%%%%%%%%%%%%%%%%%
%%%%%%%%%%%%%%%%%%%%%%%%%%%%%%%%%%%%%%%%%%%%%%%%%%%%%%%%%%%%%%%%%%%%%%%%%%%%%%%%%%%%%%%%%%%%%%%%%%%%%%%%%%%%%%%%%%%%%%%%%%%%%%%%%%%%%%%%%%%%%%%%%%%%%%%%%%%%%%%%%%%%%%%%%
%\newpage
%%%%%%%%%%%%%%%%%%%%%%%%%%%%%%%%%%%%%%%%%%%%%%%%%%%%%%%%%%%%%%%%%%%%%%%%%%%%%%%%%%%%%%%%%%%%%%%%%%%%%%%%%%%%%%%%%%%%%%%%%%%%%%%%%%%%%%%%%%%%%%%%%%%%%%%%%%%%%%%%%%%%%%%%%
%%%%%%%%%%%%%%%%%%%%%%%%%%%%%%%%%%%%%%%%%%%%%%%%%%%%%%%%%%%%%%%%%%%%%%%%%%%%%%%%%%%%%%%%%%%%%%%%%%%%%%%%%%%%%%%%%%%%%%%%%%%%%%%%%%%%%%%%%%%%%%%%%%%%%%%%%%%%%%%%%%%%%%%%%
\appendix
%%%----------------------------------------------------------------------------------------------------------------------------------------------------------------------
\setcounter{equation}{0}\numberwithin{equation}{section}
\setcounter{theorem}{0}\numberwithin{theorem}{section}
%%%----------------------------------------------------------------------------------------------------------------------------------------------------------------------
\section{H\"{o}rmander spaces on the circle}
Let $\mathrm{OR}$ be a class of all measurable by Borel functions $\omega:\;(0,\infty)\rightarrow (0,\infty)$ for which real numbers $a,c>1$ exist such that
\begin{equation*}
  c^{-1}\leq \frac{\omega(\lambda t)}{\omega (t)}\leq c\qquad \forall t\geq 1,\;\lambda\in [1,a].
\end{equation*}

The space $H_{2}^{\omega}(\mathbb{R}^{n})$, $n\in \mathbb{N}$, consists of all complex-valued distributions $u\in \mathcal{S}'(\mathbb{R}^{n})$ such that their Fourier
transformations $\widehat{u}$ are locally integrable by Lebesgue on $\mathbb{R}^{n}$ and $\omega(\langle\xi\rangle)|\widehat{u}(\xi)|\in L^{2}(\mathbb{R}^{n})$ with
$\langle\xi\rangle :=(1+\xi^{2})^{1/2}$. This space is a Hilbert space with respect to the inner product
\begin{equation*}
  \left(u_{1},u_{2}\right)_{H_{2}^{\omega}(\mathbb{R}^{n})}:=\int_{\mathbb{R}^{n}}\omega^{2}(\langle\xi\rangle)\widehat{u}_{1}(\xi)\overline{\widehat{u}_{2}(\xi)}\,d\xi.
\end{equation*}

It is a particular case of the isotropic Hilbert spaces of H\"{o}rmander \cite{Hor}. If $\Omega$ is a domain in $\mathbb{R}^{n}$ with smooth boundary, then the spaces
$H_{2}^{\omega}(\Omega)$ are defined in a standard way.

Let $\Gamma$ be an infinitely smooth, closed and oriented manifold of dimension $n\geq 1$ with density $dx$ given on it. Let $\mathfrak{D}'(\Gamma)$ be a topological
vector space of distributions on $\Gamma$ which is dual to $C^{\infty}(\Gamma)$ with respect to the extension by continuity of the inner product in the space
$L^{2}(\Gamma):=L^{2}(\Gamma,dx)$.

Now, let define the H\"{o}rmander spaces on the manifold $\Gamma$. Choose a finite atlas from the $C^{\infty}$-structure on $\Gamma$ formed by the local charts
$\alpha_{j}: \mathbb{R}^{n}\leftrightarrow U_{j},\; j=1,\ldots,r$, where the open sets $U_{j}$ form a finite covering of the manifold $\Gamma$. Let functions
$\chi_{j}\in C^{\infty}(\Gamma)$, $j=1,\ldots,r$, form a partition of unity on $\Gamma$ satisfying the condition $\mathrm{supp}\,\chi_{j}\subset U_{j}$. By definition,
the linear space $H_{2}^{\omega}(\Gamma)$ consists of all distributions $f\in \mathfrak{D}'(\Gamma)$ such that $(\chi_{j}f)\circ\alpha_{j}\in
H_{2}^{\omega}(\mathbb{R}^{n})$ for every $j$, where $(\chi_{j}f)\circ\alpha_{j}$ is a representation of the distribution $\chi_{j}f$ in the local chart $\alpha_{j}$. In
the space $H_{2}^{\omega}(\Gamma)$ the inner product is defined by the formula
\begin{equation*}
  \left(f_{1},f_{2}\right)_{H_{2}^{\omega}(\Gamma)}:=\sum_{j=1}^{r}\left((\chi_{j}f_{1})\circ\alpha_{j},(\chi_{j}f_{2})\circ\alpha_{j}\right)_{H_{2}^{\omega}(\mathbb{R}^{n})},
\end{equation*}
and induces the norm $\|f\|_{H_{2}^{\omega}(\Gamma)}:=\left(f,f\right)_{H_{2}^{\omega}(\Gamma)}^{1/2}$.

There exists an alternative definition of the space $H_{2}^{\omega}(\Gamma)$ which shows that this space
 does not depend (up to equivalence of norms) on the choice
of the local charts, the partition of unity and that it is a Hilbert space.

Let a $\Psi$DO $A$ of order $m>0$ be elliptic on $\Gamma$, and let it be a positive unbounded operator on the space $L^{2}(\Gamma)$. For instance, we can set
$A:=\left(1-\triangle_{\Gamma}\right)^{1/2}$, where $\triangle_{\Gamma}$ is the Beltrami-Laplace operator on the Riemannian manifold $\Gamma$. Redefine the function $\omega\in \mathrm{OR}$ on
the interval $0<t<1$ by the equality $\omega(t):=\omega(1)$ and introduce the norm
\begin{equation}\label{eq_ap10}
  f\mapsto \|\omega(A^{1/m})f\|_{L^{2}(\Gamma)},\qquad f\in C^{\infty}(\Gamma).
\end{equation}
\begin{theorem}\label{th_ap10}
If $\omega\in \mathrm{OR}$, then the space $H_{2}^{\omega}(\Gamma)$ coincides up to equivalence of norms with the completion of the linear space $C^{\infty}(\Gamma)$ by
the norm \eqref{eq_ap10}.
\end{theorem}

As the operator $A$ has a discrete spectrum, therefore the space $H_{2}^{\omega}(\Gamma)$ can be described by means of the Fourier series. Let
$\{\lambda_{k}\}_{k\in \mathbb{N}}$ be a monotonically non-decreasing, positive sequence of all eigenvalues of the operator $A$, enumerated with regard to their multiplicity. Let $\{h_{k}\}_{k\in \mathbb{N}}$ be an orthonormal basis in the space $L^{2}(\Gamma)$ formed by the correspondent eigenfunctions of the
operator $A$: $A h_{k}=\lambda_{k}h_{k}$. Then for any distribution the following expansion into the Fourier series converging in the linear space
$\mathfrak{D}'(\Gamma)$:
\begin{equation}\label{eq_ap12}
  f=\sum_{k=1}^{\infty}c_{k}(f)h_{k},\quad f\in \mathfrak{D}'(\Gamma),\; c_{k}(f):=(f,h_{k}),
\end{equation}
holds.
\begin{theorem}\label{th_ap12}
The following formulae are fulfilled:
\begin{align*}
 &H_{2}^{\omega}(\Gamma)  =\left\{f=\sum_{k=1}^{\infty}c_{k}(f)h_{k}\in \mathfrak{D}'(\Gamma)\;\left|\;\sum_{k=1}^{\infty}
  \omega^{2}(k^{1/n})|c_{k}(f)|^{2}<\infty\right.\right\}, \\
 &\|f\|_{H_{2}^{\omega}(\Gamma)}^{2}  \asymp \sum_{k=1}^{\infty}\omega^{2}(k^{1/n})|c_{k}(f)|^{2}.
\end{align*}
\end{theorem}

Note, that for every distribution $f\in H_{2}^{\omega}(\Gamma)$ series \eqref{eq_ap12} converges by the norm of the space $H_{2}^{\omega}(\Gamma)$. If values of the
function $\omega$ are separated from zero, then $H_{2}^{\omega}(\Gamma)\subseteq L^{2}(\Gamma)$, and everywhere above we may change the space $\mathfrak{D}'(\Gamma)$ by
the space $L^{2}(\Gamma)$. For more details, see \cite{MiMr1, MiMr2}.
\begin{example*}\label{ex_ap10}
Let $\Gamma=\mathbb{T}$. Then $n=1$, and we can choose $A=\left(1-d^{2}/dx^{2}\right)^{1/2}$, where $x$ defines the natural parametrization on $\mathbb{T}$. The
eigenfunctions $h_{k}=e^{i k2\pi x}$, $k\in \mathbb{Z}$, of the operator $A$ form an orthonormal basis in the space $L^{2}(\mathbb{T})$. For $\omega\in \mathrm{OR}$ we
have
\begin{equation*}
  f\in H_{2}^{\omega}(\mathbb{T})\Leftrightarrow f=\sum_{k\in \mathbb{Z}}\widehat{f}(k)e^{i k2\pi x},\quad \sum_{k\in\mathbb{Z}\setminus\{0\}}|\widehat{f}(k)|^{2}\omega^{2}(|k|)<\infty.
\end{equation*}
In this case the function $f$ is real-valued if and only if $\widehat{f}(k)=\overline{\widehat{f}(-k)}$, $k\in \mathbb{Z}$. Therefore \textit{the class $H^{\omega}$
coincides with the H\"{o}rmander space $H_{2}^{\omega}(\mathbb{T},\mathbb{R})$} with the weight function $\omega(\sqrt{1+\xi^{2}})$ if $\omega\in \mathrm{OR}$. In details
the class $\mathrm{OR}$ is described, for example, in \cite[p. 74]{BnGlTg}.
\end{example*}
%%%%%%%%%%%%%%%%%%%%%%%%%%%%%%%%%%%%%%%%%%%%%%%%%%%%%%%%%%%%%%%%%%%%%%%%%%%%%%%%%%%%%%%%%%%%%%%%%%%%%%%%%%%%%%%%%%%%%%%%%%%%%%%%%%%%%%%%%%%%%%%%%%%%%%%%%%%%%%%%%%%%%%%%%
%\newpage
%%%%%%%%%%%%%%%%%%%%%%%%%%%%%%%%%%%%%%%%%%%%%%%%%%%%%%%%%%%%%%%%%%%%%%%%%%%%%%%%%%%%%%%%%%%%%%%%%%%%%%%%%%%%%%%%%%%%%%%%%%%%%%%%%%%%%%%%%%%%%%%%%%%%%%%%%%%%%%%%%%%%%%%%%

\end{document}